\documentclass[11pt, reqno]{amsart}

\usepackage{ amssymb, amsmath, amsthm,  graphicx, psfrag}
\usepackage[usenames,dvipsnames]{color}
\usepackage{enumerate} 
\usepackage{algpseudocode}

\newcommand{\R}{{\mathbb R}}

\newcommand{\C}{{\mathbb C}}
\newcommand{\N}{{\mathbb N}}

\def\norm#1{\left\|#1\right\|}

\def\inpro#1{\left\langle #1 \right\rangle}
\def\Set#1{\left\{#1\right\}}
\def\gp#1{\left(#1\right)}
\def\bk#1{\left[#1\right]}

\def\calS{{\mathcal S}}
\def\calN{{\mathcal N}}
\def\calI{{\mathcal I}}

\def\bn{\begin{enumerate}}  
\def\en{\end{enumerate}}    
\def\bmt{\begin{matrix}} 
\def\emt{\end{matrix}}

\newtheorem{mythm}{Theorem}[section]
\newtheorem{theorem}{Theorem}[section]
\newtheorem{definition}[mythm]{Definition}
\newtheorem{example}[mythm]{Example}

\newtheorem{proposition}[mythm]{Proposition}
\newtheorem{lemma}[mythm]{Lemma}
\theoremstyle{remark}
\newtheorem{remark}[mythm]{Remark}
\newtheorem{observation}[mythm]{Observation}

\newcommand{\ds}{\displaystyle}
\newcommand{\f}{\{f_i\}_{i=1}^k}
\newcommand{\F}{\mathcal{F}}

\newcommand{\Hn}{\mathcal{H}_n}
\newcommand{\n}[1]{\{1,2,\ldots,#1\}}
\newcommand{\rob}[1]{\operatorname{rob}(#1)}
\newcommand{\Span}[1]{\operatorname{span}\left(#1\right)}

\begin{document}

\title{Maximum Robustness and  Surgery of Frames in finite dimensions$^*$}

\author{Martin S. Copenhaver}
\address{School of Mathematics, Georgia Institute of Technology}
\email{copenhaver@gatech.edu}

\author{Yeon Hyang Kim}
\address{Department of Mathematics, Central Michigan University}
\email{kim4y@cmich.edu}

\author{Cortney Logan}
\address{Department of Mathematics, Stonehill College}
\email{logan.cort@gmail.com}

\author{Kyanne Mayfield}
\address{Department of Mathematics, University of Portland}
\email{mayfield13@up.edu}

\author{Sivaram K. Narayan}
\address{Department of Mathematics, Central Michigan University}
\email{naray1sk@cmich.edu}

\author{Jonathan Sheperd}
\address{Department of Mathematics, University of Notre Dame}
\email{jsheperd@nd.edu}

\thanks{*Reseach supported  by NSF-REU Grant DMS 08-51321. This work was done as a part of the REU program in  Summer 2011.}


\date{February 7, 2013.}


\begin{abstract}
We consider frames in a finite-dimensional Hilbert space $\Hn$ where frames are exactly the spanning sets of the vector space. 
We present a method to determine the maximum robustness of a frame. We present results on tight subframes and surgery of frames.  We also answer the  question of when length surgery resulting in a tight frame set for $\Hn$ is possible. 
\end{abstract}

\maketitle

\section{Introduction}\label{intro}
A basis $\f$ of vectors in a finite-dimensional inner product space \( \Hn \) can be used to represent every vector  
\( f  \) as a  linear combination of the elements in  $\f$:
\[  f = \sum_i a_i f_i. \]
This representation gives us characteristics of \(f\)  in terms of the coefficients \(a_i\). 
However, uniqueness of this representation is not always an advantage. For example, in applications such as image and signal processing, the  loss of a single coefficient during  data transmission will prevent the recovery of the original signal. 

A frame is a generalization of a basis that includes  redundancy. That is, a frame in  finite dimensions is a redundant set of vectors that span a finite-dimensional  vector space.   This redundancy yields robustness, which makes frame representations less sensitive to transmission errors. 
The study of frames began in 1952 with their introduction by Duffin and Schaeffer \cite{duffin} and has since been expanded by Daubechies \cite{daubechies} and others  \cite{BF03, classes, erasures, RS95}.   

In \cite{narayan}, the authors characterize frames in $\R^n$ that are robust to $k$ erasures and a necessary and sufficient condition is given for performing $(r, k)$-surgery on unit-norm tight frames in $\R^2$. 
In this paper, we present a method to  determine the maximum robustness of any given frame, and generalize  results on surgery from \cite{narayan}.  We also answer the  question of when length surgery resulting in a tight frame set for $\Hn$ is possible. 
We begin by defining various notions that are mentioned above.  
A good introduction to frames in finite dimensions can be found in \cite{ffu, nguyen}.

Let $I$ be a subset of $\N$ and $\Hn$  be an $n$-dimensional real or
complex Hilbert space.  
A \emph{frame} in $\Hn$ 
 is a sequence of vectors $\{f_i\}_{i\in I}$ for which there exist constants $0<A\leq B<\infty$ such that for all $f \in\Hn$,
\[A\| f \|^2\leq\displaystyle\sum_{i\in I}|\langle f,f_i \rangle|^2 \leq B\|f\|^2.\]
When $A=B=\lambda$, $\{f_i\}_{i\in I}$ is called a \emph{$\lambda-$tight frame}.  When  $\lambda=1$,  the frame is called a \emph{Parseval frame}. A \emph{unit-norm frame} is a frame such that each vector in the frame has norm one. In a finite-dimensional Hilbert space $\Hn$, a sequence of vectors is a frame if and only if it spans  $\Hn$.

Given a sequence of vectors  $\Set{f_i}_{i=1}^k$ in $\Hn$,  we define the \emph{analysis operator} to be the linear map $\theta: \Hn \rightarrow \ell^2(\Set{1, \cdots, k})$ defined by $(\theta f)(i)=\langle f, f_i \rangle$. The adjoint $\theta^*$ such that $\theta^*\,:\, \ell^2(\Set{1, \cdots, k})\rightarrow \Hn$ is called the \emph{synthesis operator}. In $\Hn$, the analysis operator associated with a sequence of vectors $\f$ can be written  with respect to a basis as  the $k\times n$ matrix
\[
\theta = \left[
\begin{array}{c c c }
\leftarrow & f_1^*& \rightarrow \\
& \vdots & \\
\leftarrow & f_k^*& \rightarrow 
\end{array}\right]
,\]
and the synthesis operator as the $n\times k$ matrix
\[\theta^* = \left[
\begin{array}{c c c }
\uparrow & & \uparrow \\
f_1 & \cdots & f_k\\
\downarrow & & \downarrow
\end{array}\right].
\]

The \emph{frame operator} $S$ of a sequence of vectors $\Set{f_i}_{i=1}^k$ (not necessarily a frame) is defined as $\theta^*\theta$. For all $f\in\Hn$,
\[
Sf=\theta^*\theta f= \displaystyle\sum_{i =1}^k \langle f,f_i\rangle f_i
.\]
For $\f\subseteq  \Hn$, the \emph{Gramian operator} $G$ is the $k\times k$ matrix defined by
\[
G=\theta\theta^*=(\langle f_j, f_i \rangle)_{i,j=1}^k.
\]
Given a sequence of vectors $\f$ in $\Hn$, it is known that the frame operator $S$ of the sequence has rank $n$ if and only if the sequence is a frame. The frame operator  $S= \lambda I_n$ if and only if $\f$ is a $\lambda-$tight frame. Moreover, $S = I_n$ if and only if $\f$ is a Parseval frame \cite{ffu}.

\section{Robustness}

In this section, we present a method to determine the maximum robustness of any given frame.  

\begin{definition}\label{Jdefn:robust}
A frame \(\F := \{f_i\}_{i=1}^k\) for \(\Hn\) is said to be \emph{robust to \(r\) erasures} if \(\{f_i\}_{i\in I^C}\) is still a frame for any index set \(I\subseteq \{1,2,\ldots,k\}\) with \(|I|=r\). 
The \emph{maximum robustness} of a frame \(\mathcal{F}\) for \(\mathcal{H}_n\), which we denote \(\operatorname{rob}(\mathcal{F})\), is defined to be the maximum number \(r\) such that \(\mathcal{F}\) is robust to \(r\) erasures.
\end{definition}

\begin{observation} \label{upper}
From the definition it follows that \(\text{rob}(\F) \le k-n. \) 
In the case when $\F$ is a full spark frame, i.e., every set of $n$ vectors in the frame $\F$ form a basis for $\Hn$ \cite{ACM12},  we have 
$\rob{\F} = k-n . $
\end{observation}

The goal is to develop methods to easily compute 
 the maximum robustness of any given frame. There are some known results that allow us to check whether a frame is robust to a particular number of erasures. One obvious method to determine maximum robustness is to use these results to find the number \(r\) such that the frame is robust to \(r\), but not \(r+1\), erasures. 
 We first state two known results.

\begin{theorem}[\cite{erasures}]\label{Jthm:robTo1}
Let \(\mathcal{F}=\{f_i\}_{i=1}^k\) be a frame for  \(\mathbb{R}^n\). The following are equivalent:
\item[(1)] \( \F \) is a frame robust to one erasure.
\item[(2)] There are scalars \(c_i\neq 0\), for \(1\le i\le k\), such that
\begin{equation*}
\sum_{i=1}^k c_i f_i=0.
\end{equation*}
\end{theorem}

The following is a generalization of Theorem \ref{Jthm:robTo1}.

\begin{theorem}[\cite{narayan}]\label{Jthm:robToR}
Let \(\mathcal{F}=\{f_i\}_{i=1}^k\) be a frame for \(\mathbb{R}^n\) with synthesis operator \(\theta^*\). The following are equivalent:
\item[(1)]  \(\F\) is a frame robust to \(r\) erasures.
\item[(2)]  For all index sets \( \calI \subseteq  \{1,2,\ldots,k\}\) with \(| \calI |=r-1\), 
\[ \calI^C\in \Set{ supp(f) \,:\, f \in  null (\theta^*)} ,\]
where \(supp(f) \) is the set of indices where the vector $f$ has nonzero components.  
\end{theorem}

In order to use Theorem \ref{Jthm:robToR} we should be able to compute the support of the null space of the synthesis operator. We have provided an algorithm for computing the support of the null space of a matrix $A$ in the Appendix. 
For more details about erasure we refer the reader to \cite{HP04, PK05, FM12}. 
The following lemma states that for finding maximum robustness of a frame we can ignore zero vectors in the frame. 

\begin{lemma}\label{twothree}
Let \(\mathcal{F}=\{f_i\}_{i=1}^k\) be a frame for \(\mathcal{H}_n\).  Then $\rob{\F} = \rob{\F \setminus \{0\}}$. 
\end{lemma}
The proof  of Lemma \ref{twothree}  follows from the observation that for any subset $\calS$ of frame vectors,  
\( \text{span}(\calS) = \text{span}(\calS \setminus \{0\} ) \). 
The next proposition gives an upper bound for \( \rob{\F} \). 

\begin{proposition}\label{Jprop:upperBound}
Let \(\mathcal{F}=\{f_i\}_{i=1}^k\) be a frame for \(\mathcal{H}_n\) and let 
\[\calS=\{I\subseteq\{1,2,\ldots,k\}:\operatorname{span}(\{f_i\}_{i\in I})=\mathcal{H}_n\}\] be the collection  of index sets that correspond to spanning sets. Then \(\operatorname{rob}(\mathcal{F})\le \lfloor\log_2{|\calS|}\rfloor\).
\end{proposition}
\begin{proof}
Suppose \(\mathcal{F}\) is robust to  \( m = \lfloor\log_2{|\calS|}\rfloor+1\) erasures. 
Then \( \{f_i\}_{i=m+1}^k  \) is a frame. Let 
\[\calS'=\{I\subseteq\{m+1, \ldots,k\}:\operatorname{span}(\{f_i\}_{i\in I})=\mathcal{H}_n\}.\] 
Then for any  $J \subseteq \Set{1, 2, \cdots, m}$ and $I \in \calS'$, $I \cup J  \in \calS$. 
Thus we have $|\calS| \ge 2^m |\calS'|$, which implies 
\[ |\calS'| \le \frac{|\calS|}{2^m} <
\frac{|\calS|}{2^{\log_2{|\calS|} }} =1.  \]
Consequently,  $\calS' = \emptyset$. This is a contradiction since 
$\Set{m+1, \cdots, k} \in \calS'.$
\end{proof}

We note that if $\mathcal{F}$ is a frame robust to $r$ erasures, then $ r \le \operatorname{rob}(\mathcal{F})$ 
and $$ \sum_{i=0}^r \left(  \begin{matrix} n  \cr n-r+i \end{matrix} \right)  \le |\calS| $$ since each set of at least $n-r$ vectors is a frame. 
In the next proposition, we express the maximum robustness of a frame  using the 
cardinality of maximum  nonspanning set. 

\begin{proposition}\label{Jprop:HnMaxRob}
Let \(\mathcal{F}=\{f_i\}_{i=1}^k\) be a frame for \(\mathcal{H}_n\), and let \(N=\{I\subseteq\{1,2,\ldots,k\}:\operatorname{span}(\{f_i\}_{i\in I})\neq\mathcal{H}_n\}\) be the collection of index sets that correspond to  nonspanning sets. Then
\begin{equation*}
\operatorname{rob}(\mathcal{F})=k-1-\max\{|I|:I\in N\}.
\end{equation*}
\end{proposition}

\begin{proof}
Let \(I_0\in N\) be such that \(|I_0|=\max\{|I|:I\in N\}\). Suppose any \(k-1-|I_0|\) vectors are removed from \(\mathcal{F}\). Then \(|I_0|+1\) vectors remain. Since the remaining set spans $\Hn$,   \(\mathcal{F}\) is robust to \(k-1-|I_0|\) erasures. 
Suppose the \(k-|I_0|\) vectors whose indices are not in \(I_0\) are removed from \(\mathcal{F}\). The remaining vectors corresponding to \(I_0 \) do not span \(\mathcal{H}_n\), so \(\mathcal{F}\) is not robust to \(k-|I_0|\) erasures. Therefore, 
\(
\operatorname{rob}(\mathcal{F})=k-1-|I_0|.
\)
\end{proof}

Using Proposition \ref{Jprop:HnMaxRob} we can find  a vector $y \in \Hn$ so that the  maximum robustness of a frame   
\(\mathcal{F}=\{f_i\}_{i=1}^k\)   in $\Hn$ is equal to one less than the count of nonzero numbers in the set $\{\langle f_i, y \rangle\}_{i=1}^k$. This is stated in the next theorem. 

\begin{theorem}\label{mrankoneproj}
If \(\mathcal{F}=\{f_i\}_{i=1}^k\) is a frame for \(\mathcal{H}_n\), then there exists a $y \in \Hn$ such that 
 $$\rob{\F} = \rob{\{\langle  f_i,y\rangle \}_{i=1}^k}.$$
\end{theorem}

\begin{proof} 
Let $\calN$ be a largest non-spanning subset of $\f$.   
Then we have 
$\rob{\f}  =  k -1 -|\calN|$ by Proposition \ref{Jprop:HnMaxRob} and $\Span{\calN}^\perp = \Span{y}$ for some $y$ in \(\mathcal{H}_n\). 
Since $ \inpro{f_i, y}=0$ if and only if  $f_i\in \calN $,  $\{\langle f_i, y \rangle\}_{i=1}^k$ contains $|\calN|$ zero vectors and $k-|\calN|$ nonzero vectors, which implies that 
$\rob{\{\langle f_i, y\rangle\}_{i=1}^k} =  k -1 -|\calN|$.
\end{proof}

The following is an example of the implementation of this theorem. We note that in general finding the largest
non-spanning set is a combinatorially hard problem. 

\begin{example}
Let $$\mathcal{F} = \left\{\left(\begin{array}{c}
3\\
0
\end{array}\right),
\left(\begin{array}{c}
2\\
0
\end{array}\right),
\left(\begin{array}{c}
1\\
0
\end{array}\right),
\left(\begin{array}{c}
1\\
1
\end{array}\right),
\left(\begin{array}{c}
0\\
1
\end{array}\right),
\left(\begin{array}{c}
1\\
-1
\end{array}\right)\right\}\subseteq \R^2.$$ The largest non-spanning set  is $ \calN = \left\{\left(\begin{array}{c}
3\\
0
\end{array}\right),
\left(\begin{array}{c}
2\\
0
\end{array}\right),
\left(\begin{array}{c}
1\\
0
\end{array}\right)\right\}$.
As in the proof of Theorem 2.3,  we take $y = \left(\begin{array}{c}
0\\
1
\end{array}\right)$. Thus, taking the inner product of vectors in $\mathcal{F}$ with $y$, we have $\{0,0,0,1,1,-1\}$, which clearly has a maximum robustness of two. Therefore $\mathcal{F}$ must have a maximum robustness of two as well.
\end{example}

The following theorem tells us that a transformation onto a smaller space can also preserve the maximum robustness of a particular frame.

\begin{theorem}\label{Jthm:robPreservingProj}
Let \(\F=\{f_i\}_{i=1}^k\) be a frame for \(\Hn\) with  \( \rob{\F} = r\) and let \(P\) be an orthogonal projection onto a subspace \(U\). The frame \(  P\F \) for \( U \) has maximum robustness \(r\) if and only if  there exists a subset \(\mathcal{F}'\) of \(k-r-1\) vectors which does not span \(\Hn\) and 
\(\operatorname{span}(\mathcal{F}')^\perp\subseteq  U\). 
\end{theorem}

\begin{proof}
(\(\Longrightarrow\))  Without loss of generality, we assume that \(\{P(f_i)\}_{i=1}^{k-r-1}\) does not span \(U\). Then \(\mathcal{F}'=\{f_i\}_{i=1}^{k-r-1}\) does not span \(\Hn\). 
Since \( \rob{\F} = r\), \(\operatorname{span}(\mathcal{F}')^\perp\) has dimension one.  
Since $P\F'$ does not span $U$, there exists a nonzero vector $e\in U$ such that 
$\inpro{e, Pf} = \inpro{e, f}=0,$ for $ f \in \F'$, which implies that  
\(\operatorname{span}(\mathcal{F}')^\perp = \Span{e} \subseteq U.\)  

(\(\Longleftarrow\)) Suppose there is some set \(\mathcal{F}'\) of \(k-r-1\) vectors from \(\mathcal{F}\) such that \(\mathcal{F}'\) does not span \(\Hn\) and \(\operatorname{span}(\mathcal{F}')^\perp\subseteq U\). Then for any nonzero vector \(e\in\operatorname{span}(\mathcal{F}')^\perp\) and \(f\in\mathcal{F}'\), \(0=\langle e,f\rangle=\langle e,Pf \rangle\) since \(e\in U\). Thus  \(\operatorname{span}(\mathcal{F}')^\perp\) has dimension at least one. Therefore \(P(\mathcal{F}')\) does not span \(U\), which implies that  \(P(\mathcal{F})\) is not robust to \(r+1\) erasures. Since the projected frame is necessarily robust to \(r\) erasures,  its maximum robustness is \(r\). \end{proof}

\begin{definition}\cite{redundancy}
Let \(\mathcal{F}=\{f_i\}_{i=1}^k\) be a frame for \(\mathcal{H}_n\).   For each $x \in \mathbb{S}$, the \textit{redundancy function} $\mathcal{R}_{\F} : \mathbb{S} \rightarrow \mathbb{R}^+$ is defined by
\[ \mathcal{R}_{\F}\left(x\right) = \sum_{i=1}^k\left\|P_{\left\langle f_i\right\rangle}\left(x\right)\right\|^2,
\]
where $\mathbb{S} = \left\{x \in \Hn : \left\|x\right\| = 1\right\}$  is the unit sphere in $\Hn$ and $P_{\left\langle f_i\right\rangle}\left(x\right)$ is the orthogonal projection of $x$ onto the span of $f_i$.
\end{definition}

In \cite{redundancy}, the concepts of upper and lower redundancy are also defined. The \textit{upper redundancy} $\mathcal{R}_{\F}^+$ of a frame is the maximum of its redundancy function taken over the unit sphere $\mathbb{S}$.  The \textit{lower redundancy} $\mathcal{R}_{\F}^-$ of a frame is the minimum of its redundancy function taken over the unit sphere $\mathbb{S}$.  According to \cite{redundancy}, $\left\lfloor \mathcal{R}_{\F}^-\right\rfloor$ is the maximum number of disjoint spanning sets in the frame $\F$ and $\left\lceil \mathcal{R}_{\F}^+\right\rceil$ is the minimum number of disjoint linearly independent sets in the frame $\F$. 

If a frame $\F$ in $\mathbb{R}^n$ has a maximum of $\left\lfloor \mathcal{R}_{\F}^-\right\rfloor$ disjoint spanning sets,  then by removing a vector from all but one of the disjoint spanning sets  
the frame is robust to at least $\left\lfloor \mathcal{R}_{\F}^-\right\rfloor - 1$ erasures. This observation together with Proposition \ref{Jprop:upperBound} gives us both 
 a lower and an upper bound for maximum robustness of a frame, namely 
$$ \left\lfloor \mathcal{R}_{\F}^-\right\rfloor -1 \leq \rob{\F} \le \text{min} \{ \lfloor\log_2{|\calS|}\rfloor, k-n\}.$$
\section{Tight subframes and Surgery on frames}
While there exist constructions that take a frame $\f \subseteq \Hn$ and add vectors to produce a tight frame \cite{classes}, it is not  clear whether we can instead choose some subset $\Set{f_i}_{i \in I} \subseteq \f$ that is a tight frame? This is useful to consider for several reasons. We may be interested in having vectors of certain norms. While we may apply some method of construction that could possibly keep the norms of the added vectors within a specified range, removing vectors from a frame to produce a tight frame would leave the norms unaffected. This ultimately relies only on what we have, so it is possible that this \emph{trimming method} may preserve some special features of the initial frame.

\begin{definition}\label{defsub}
Given a frame $\f $ for $ \Hn$, we say that $\Set{f_i}_{i \in I}  \subseteq \f$ is a \emph{tight subframe} if $\Set{f_i}_{i \in I} $ is itself a tight frame for $ \Hn$.
We say that a  \emph{$(p,q)$-surgery on $\f$ is possible} if and only if there exist $I \subseteq  \n{k}$ with $|I| = p$ and  $\{g_j\}_{j=1}^q \subseteq \Hn$ so that $\{f_i\}_{i \in I^C} \cup \{g_j\}_{j=1}^q$ is a tight frame for $\Hn$. When the relevant set $\f$ is clear, we may simply say that $(p,q)$-surgery is possible or impossible. In the case of surgery on unit-norm frames, we require that the new collection $\{g_j\}_{j=1}^q \subseteq \Hn$ contains only unit-norm vectors.
\end{definition}

\begin{definition}[\cite{REU11}]\label{Jdefn:RnDV}
For any vector \(f\in\mathbb{R}^n\), we define the diagram vector associated with \(f\), denoted \(\tilde{f}\), by
\begin{equation*}
\tilde{f}= 
\frac{1}{\sqrt{n-1}}
\begin{bmatrix}f^2(1)-f^2(2)\\  \vdots  \\ f^2(n-1)-f^2(n) \\  
\sqrt{2n}f(1)f(2) \\ \vdots  \\  \sqrt{2n}f(n-1)f(n)
 \end{bmatrix}\in\mathbb{R}^{n(n-1)},
\end{equation*}
where the difference of squares 
$f^2(i)- f^2(j)$ and the 
 product \(f(i)f(j)\)  occur exactly once for \(i < j, \ i = 1, 2, \cdots, n-1.\) 
 
 For any vector \(f\in\mathbb{C}^n\), we define the diagram vector associated with \(f\), denoted \(\tilde{f}\), by
\begin{equation*}
\tilde{f}= 
\frac{1}{\sqrt{n-1}}
\begin{bmatrix}f(1) \overline{f(1)}-f(2)\overline{f(2)} \\  \vdots  \\ f(n-1)\overline{f(n-1)}-f(n)\overline{f(n)} \\  
\sqrt{n}f(1) \overline{f(2)} \\ \sqrt{n} \overline{f(1)} f(2) \\ \vdots  \\  \sqrt{n}f(n-1)\overline{f(n)} 
\\ \sqrt{n} \overline{f(n-1)} f(n)
 \end{bmatrix}\in\mathbb{C}^{3n(n-1)/2},
\end{equation*}
where the difference of the form 
$f(i) \overline{f(i)} - f(j) \overline{f(j)}$ occurs exactly once for \(i < j, \ i = 1, 2, \cdots, n-1\)  and the 
 product of the form \(f(i)  \overline{f(j)} \)  occurs exactly once for \(i  \neq j.\) 
 \end{definition}

In order to give a formulation of  tight subframes and surgeries, we present the following characterization of tight frames.  
  We use the following remark in the next proposition. 
\begin{remark} (\cite{ffu}, p.121) \label{sym2}
Let $A, B$ be self-adjoint positive semidefinite matrices such that 
$\inpro{Ax, x} = \inpro{Bx, x}$ for all $ x \in \Hn$ . Then  $A=B$.
\end{remark}

\begin{proposition}\label{tight_char}
Let $k\ge n, \ \lambda >0$ and  $\F = \Set{f_i}_{i=1}^k$ be a sequence of vectors in $ \R^n ( \text{or } \C^n)$, not all of which are zero.  
Let $G$ be the Gramian associated with $\F$. 
Let $\tilde{G}$ be the Gramian associated with the diagram vectors 
 $\{\tilde{f}_i\}_{i=1}^k$. 
 Then  the following conditions are equivalent.
 \bn
\item[(1)] $\F$ is a  $\lambda$-tight  frame.
\item[(2)] $G$ has rank $n$ and $G^2 = \lambda G$.
\item[(3)] $\sigma (G) = (\underbrace{\lambda,\ldots,\lambda}_{n \text{ times}},\underbrace{0,\ldots,0}_{k-n \text{ times}})$.
 \item[(4)]  \(\sum_{i=1}^k\tilde{f_i}=0\). 
\item[(5)] $(1, 1, \cdots, 1)^T \in null(\tilde{G})$. 

\en
\end{proposition}
\begin{proof}
Let $\theta^*$ and $S=\theta^* \theta$ be the synthesis and frame operators corresponding to $\Set{f_i}_{i=1}^k $, respectively. We recall that $S = \lambda I_n$ if and only if $\F$ is  a $\lambda$-tight frame.\\
(1) $\Rightarrow$ (2)
Since $\theta^* \theta = \lambda I_n$, we have  $\theta \theta^* \theta \theta^* = \theta \lambda I_n \theta^*$, which is equivalent to $G^2 = \lambda G$. Also, $rank(G) = rank(S) = n$.\\
(2) $\Rightarrow$ (1)
Since $rank(S) = rank(G) =n$ and $S$ is an $n\times n$ matrix, 
$S$ is invertible. Hence  $\theta^*$ is onto. 
Since $ G^2 = \lambda G $ is equivalent to   $\theta S \theta^* = \theta  \lambda I_n \theta^*$, we have  
$(\theta^*x)^* S  (\theta^*x) = (\theta^*x)^* \lambda I_n (\theta^*x)$ for all $ x \in \R^n ( \text{or } \C^n). $ 
By Remark \ref{sym2}, we conclude $S = \lambda I_n$.\\
(1) $\iff$ (3)
Since $S=\theta^* \theta$ is an $n\times n$ matrix and $G=\theta \theta^*$ is a $k \times k$ matrix, the result now  follows from $\sigma(S)  \cup (\underbrace{0,\ldots,0}_{k-n \text{ times}} )  = \sigma(G)$. \\
(1) $\iff$ (4)  follows from Proposition 2.4 and 2.8 in  \cite{REU11}. \\
(1) $\iff$ (5)  follows from Proposition 4.2  in  \cite{REU11}. \\
\end{proof}

The following result gives a necessary condition for the existence of tight subframes. 
For more details about tight subframes we refer the reader to \cite{LMO12}.  

\begin{proposition}\label{mtightsubframekless2n}
Let \(\mathcal{F}=\{f_i\}_{i=1}^k\) be a  tight frame for $\R^n (\text{or }\C^n)$. 
If \(\F\) has a tight subframe, then \(k \ge 2n\). 
\end{proposition}
\begin{proof}
If there exists a tight subframe $\{f_i\}_{i \in I} \subseteq \F$, then since 
\(\sum_{i=1}^k\tilde{f_i}=0\) and \(\sum_{i\in I} \tilde{f_i}=0\), by condition (4) of Proposition  \ref{tight_char}, 
\( \{f_i\}_{i \in I^C} \) is also a tight subframe. This implies that 
\( k = |I| +|I^C| \ge 2n \) since each subframe must span $\Hn$. 
\end{proof}

We also observe that if $\f $ is a tight frame for \(\mathbb{R}^n ( \text{or } \C^n) \) with $n \ge 2$ and $k\ge n$, then nontrivial $(0,r), (r,0)$-surgery is impossible for $r = 1,2,\ldots,n-1$. 
The next proposition follows from this observation and  condition (5) of Proposition  \ref{tight_char}.  

%

\begin{proposition}\label{pro36}
If $\F = \f$ is a unit-norm  frame for $\R^n (\text{or }\C^n)$, then the following conditions are equivalent.
\begin{enumerate}
\item There exists a unit-norm tight subframe $\Set{f_j}_{j \in I}  \subseteq \F$ where 
\( I = \{ i_1, \cdots, i_m\}\).
\item An $(r,0)$-surgery on $\F$ which leaves a unit-norm tight frame is possible for some $r \in \{1, 2, \ldots,k-n\}$. 
\item Each row sum of  the Gramian of $\{\widetilde{f}_{i_1}, \widetilde{f}_{i_2},\ldots, \widetilde{f}_{i_m}\} $ is zero.
\end{enumerate}
\end{proposition}

Next, we provide a generalization of the necessary condition for $(p,q)$-surgery presented in \cite{narayan}.

\begin{theorem}
Suppose $\F = \f$ is a unit-norm tight frame for $\R^n (\text{or }\C^n)$. If a $(p,q)$-surgery on $\F$ which leaves a unit-norm tight frame is possible then the sum of the entries of the Gramian $(\langle \widetilde{f}_j, \widetilde{f}_i \rangle )_{i,j=1}^{k-p}$ is bounded above by $q^2$, where 
$\widetilde{f}_1,\ldots,\widetilde{f}_{k-p}$ denote the diagram vectors that remain after removing $p$ vectors from $\F$. 
\end{theorem}

\begin{proof}
Let $W = \widetilde{f}_1 + \cdots + \widetilde{f}_{k-p}.$ If a $(p,q)$-surgery on $\F$ is possible, then $\|W\| \leq q$ since $  \| \widetilde{f}_i \|= \norm{f}$. Note that
\begin{align*}
\|W\|^2 &= \langle \widetilde{f}_1 + \cdots + \widetilde{f}_{k-p}, \widetilde{f}_1 + \cdots + \widetilde{f}_{k-p} \rangle\\
&= \sum_{i, j=1}^{k-p} \langle \widetilde{f}_i ,\widetilde{f}_j\rangle.
\end{align*}

Thus, if $(p,q)$-surgery is possible, then the sum of the entries of the Gramian $(\langle \widetilde{f}_j, \widetilde{f}_i \rangle )_{i,j=1}^{k-p}$ is bounded above by $q^2$.
\end{proof}

We conclude our results on general surgeries on frames with the following theorem.

\begin{theorem}\label{mpqsurg1}
If  $(p,q)$-surgery on a frame  $\f \subseteq \Hn$ is possible, then $(r,r-p+q)$-surgery is possible for $r = p+1,p+2,\ldots,k$.
\end{theorem}

\begin{proof}
Suppose that  $(p,q)$-surgery on $\f$ is possible. Without loss of generality, let \( \{f_1,\ldots,f_{k-p} \}\) be the set of   vectors remaining after  removal of $p$ vectors  and let $\{g_1,\ldots,g_q \}$ be the $q$ vectors added to $\{f_1,\ldots,f_{k-p}\}$ so that $\{f_i\}_{i=1}^{k-p} \cup \{g_j\}_{j=1}^q$ is a tight frame. Let $r \in \{p+1,p+2,\ldots,k\}$. Then $(r,r-p+q)$-surgery on a frame $\f \subseteq \Hn$  is possible, by excising \(\{f_{k-r+1},f_{k-r+2},\ldots,f_k\}\), 
and then adding back the $r-p + q $ vectors $\{f_{k-r+1},f_{k-r+2},\ldots,f_{k-p}\} \cup \{g_j\}_{j=1}^q$ so that
$$\{f_1,f_2,\ldots,f_{k-r}\} \cup \{f_{k-r+1},f_{k-r+2},\ldots,f_{k-p}\} \cup \{g_j\}_{j=1}^q $$
$$= \{f_i\}_{i=1}^{k-p} \cup \{g_j\}_{j=1}^q,$$
 which is tight.
\end{proof}

\section{Length surgery}

While frame surgery involves adding vectors to a frame or removing vectors from a frame, 
length surgery deals only with  the norms of  vectors in a frame.

\begin{definition}[\cite{narayan}]
\label{kLengthSurgDef}
Let $a_1$, $a_2$, $\ldots$, $a_k$ denote norms of $k$ vectors in $\Hn$.  A $\left(p, q\right)$-length surgery on $\left\{a_i\right\}^{k}_{i=1}$ removes $p$ numbers and replaces them with $q$ nonnegative numbers.
\end{definition}

An interesting question while  performing length surgery is whether the norms corresponding to the set $\left\{a_i\right\}^{k}_{i=1}$ or any modified set are norms of vectors that form a tight frame. 

\begin{definition}
\label{kTightFrameSetDef}
A set of numbers $\left\{a_i\right\}_{i=1}^k$ is a \textit{tight frame set} for $\Hn$ if $\left\{a_i\right\}_{i=1}^k$ contains the norms of vectors of a tight frame in $\Hn$.
\end{definition}

The following theorem describes exactly when a set of numbers is  a tight frame set for $\Hn$.

\begin{theorem}[\cite{physint}]
\label{kFundIneqThm}
Given an $n$-dimensional Hilbert space $\mathcal{H}_n$ and a sequence of positive scalars $\{a_i\}_{i=1}^k$, there exists a tight frame $\{f_i\}_{i=1}^k$ for $\mathcal{H}_n$ of lengths $\|f_i\|=a_i$ for all $i=1,\ldots,k$ if and only if
\begin{equation}\label{eq:fundIneq}
\max_i a_i^2 \leq \frac{1}{n}\sum_{i=1}^{k}a_i^2.
\end{equation}
\end{theorem}

Due to its profound importance, the inequality \eqref{eq:fundIneq} is often called the \emph{fundamental inequality} \cite{ffu}. 
An immediate consequence of Theorem \ref{kFundIneqThm} is that 
$\left\{a_i\right\}_{i=1}^n$ is a tight frame set for $\Hn$ if and only if 
$a_1 = \cdots = a_n$. 

When performing $(p, q)$-length surgery, we would like to know when length surgery resulting in a tight frame set for $\Hn$ is possible, and furthermore, specifically what  numbers can be added to a set of nonnegative  numbers  for it to become or remain a tight frame set.  

\begin{proposition}\label{rmk1}
Suppose that  $\left\{a_i\right\}_{i=1}^k$ is a tight frame set for $\Hn$ and let $ m_q:= \max_{i=q+1, \cdots, k} a_i^2 $ for $q=1,2,\ldots,k-1$. Then we have that 
$$  m_q \leq \frac{1}{n-q}\sum_{i=q+1}^{k}a_i^2.$$
In particular, we have
$$ nm_1 -  \sum_{i=2}^{k}a_i^2 \le m_1 \leq \frac{1}{n-1}\sum_{i=2}^{k}a_i^2.$$
\end{proposition}
\begin{proof}
By  Theorem \ref{kFundIneqThm}, we have 
\( \max_i a_i^2 \leq \frac{1}{n}\sum_{i=1}^{k}a_i^2\), which is equivalent to 
\( n \gp{ \max_i a_i^2} -  \sum_{i=1}^{q}a_i^2  \le  \sum_{i=q+1}^{k}a_i^2  \). 
Since 
\( n\gp{ \max_i a_i^2}-\sum_{i=1}^{q}a_i^2 \ge n\gp{ \max_i a_i^2} - q \gp{ \max_i a_i^2} \ge (n-q)m_q  \), we have 
\( m_q \le   \  \frac{1}{n-q}\sum_{i=q+1}^{k}a_i^2\).  
The left inequality for $m_1$ is  equivalent to 
\( m_1 \le   \  \frac{1}{n-1}\sum_{i=2}^{k}a_i^2\). 
\end{proof}

In the following theorem we state what nonnegative  numbers \(b\) should replace $a_1$ in a tight frame set 
$\left\{a_i\right\}_{i=1}^k$ so that it remains a tight frame set. 

\begin{theorem}\label{pro_tfs}
Suppose that $\left\{a_i\right\}_{i=1}^k$ is a tight frame set for $\Hn$ and $b$ is a nonnegative  scalar. Then $\Set{b, a_2, \cdots, a_k}$ is a tight frame set for $\Hn$ if and only if $ b^2 \in \bk{  nm - \sum_{i=2}^{k}a_i^2, \  \frac{1}{n-1}\sum_{i=2}^{k}a_i^2}   $, where   $ m:= \max_{i=2, \cdots, k} a_i^2. $
\end{theorem}
\begin{proof}
Let us assume $b^2 \le m$. Now,  
$ b^2 \ge nm - \sum_{i=2}^{k}a_i^2 $  is equivalent to  $ m \le \frac{1}{n} \gp{b^2 +  \sum_{i=2}^{k}a_i^2 }$. This is equivalent to  $\Set{b, a_2, \cdots, a_k}$ being a tight frame set for $\Hn$ by Theorem \ref{kFundIneqThm} since 
$ m = \max  \Set{b^2, a_2^2, \cdots, a_k^2}$. 
Suppose $b^2 >  m$. Now,  
$  b^2 \le  \frac{1}{n-1}\sum_{i=2}^{k}a_i^2 $, which is equivalent to  $ b^2 \le \frac{1}{n} \gp{b^2 +  \sum_{i=2}^{k}a_i^2 }$.  This is equivalent to  $\Set{b, a_2, \cdots, a_k}$ being a tight frame set for $\Hn$ by Theorem \ref{kFundIneqThm} since 
$ b^2 = \max  \Set{b^2, a_2^2, \cdots, a_k^2}$. \end{proof}

The previous theorem states when $\left(1, 1\right)$-length surgery resulting in a tight frame set is possible.   
The next theorem describes when $\left(p, q \right)$-length surgery resulting in a tight frame set is possible.   
This result is a restatement of Theorem 4.9 in \cite{narayan} which gives a condition for $\left(0, q \right)$-length surgery resulting in a tight frame set; the proof given here is different from that of \cite{narayan}.

\begin{theorem}
\label{p1_tfs} 
Let $n \geq 2$, $0 \leq p \leq k $ and $\{a_i\}_{i=1}^k$ be a sequence of positive scalars.  Then 
$\left(p, q \right)$-length surgery on $\left\{a_i\right\}_{i=1}^{k}$ resulting in a tight frame set is possible if and only if there exists a subset $\left\{b_j\right\}_{j=1}^{k-p}$ of $\left\{a_i\right\}_{i=1}^{k}$ with 
$ \max_{1\le j \le k-p}  b_j^2  \leq \frac{1}{n-q}\sum_{j=1}^{k-p} b_j^2. $
\end{theorem}
\begin{proof}
($\Rightarrow$) The set after $\left(p, q\right)$-length surgery on $\left\{a_i\right\}_{i=1}^{k}$  satisfies the desired inequality from Proposition \ref{rmk1} by letting $\Set{a_i}_{i=1}^q$ be the set of the numbers which are added to the set of \(k-p\) remaining numbers. 

($\Leftarrow$) Choose $ b_{k-p+i} := \gp{\max_{1\le j \le k-p}  b_j^2 }^{1/2}$ for \(i =1, \cdots, q\). \\
Then $ \max_{1\le j \le k-p}  b_j^2  \leq \frac{1}{n-q}\sum_{j=1}^{k-p} b_j^2 $ implies that 
$  (n-q) \max_{1\le j \le k-p}  b_j^2  \leq \sum_{j=1}^{k-p} b_j^2 $. By adding $ \sum_{j=1}^{q} b_{k-p+i}^2 $ to  both sides of this inequality, we obtain the fundamental inequality for $\{b_j\}_{i=1}^{k-p+q}$. 
\end{proof}

In the Appendix we provide an algorithm for computing the support of the null space of a matrix $A$. 
We note that this algorithm has an exponential complexity in the number of vectors in a frame.

\section*{Acknowlegement}
Copenhaver, Logan, Mayfield, Narayan, and Sheperd were supported by the NSF-REU Grant DMS 08-51321.
 Kim was supported by the Central Michigan University ORSP Early Career Investigator
(ECI) grant \#C61373. The authors thank M. Petro who provided the algorithm for computing the support of the null space of a matrix $A$.

\section*{Appendix}
\renewcommand{\algorithmiccomment}[1]{ \# #1}
\begin{algorithmic}
\State \Comment{Computes the support of the null space of an \(n\times m\) matrix $A$}
\State  \Comment{Initially invoke as \Call{support}{$A$,\(\Set{ }\),$m$}}
\State  \Comment{Output: a subset of the power set of \(\Set{1, \cdots, m}\)}

\Function {support}{$A$, $M$, $j$}
\If {$j=0$} \Comment{base case}
   \State {\bf let} $tA$ be the submatrix of $A$ consisting of the columns 
   \State     \qquad  with indices in $M$
   \State {\bf choose} $\Set{v_1, \cdots, v_l}$ spanning $null(tA)$ \Comment{using SVD}
   \If{$\ds \bigcup_{i=1}^l  supp(v_i) = \Set{1,  \cdots  ,|M|}$}
       \State  \Return $\Set{M}$
   \Else 
        \State  \Return $\Set{}$
   \EndIf
\Else \Comment{recursive case}
\State \Return \Call{support}{$A$,\(M\cup \Set{j} \), $j-1$} $\cup$ \Call{support}{$A$,\(M\),$j-1$}
\EndIf
\EndFunction
\end{algorithmic}

\bibliographystyle{amsplain}
\bibliography{References}
\end{document}